\documentclass[12pt,a4paper]{article}
\pdfoutput=1
\usepackage{amsthm}
\usepackage{newtxtext,newtxmath}
\usepackage{enumerate}
\usepackage[margin=3.2cm]{geometry}
\usepackage{authblk}
\usepackage{tikz}
\usetikzlibrary{calc}
\usepackage[colorlinks=true]{hyperref}
\providecommand{\lBrack}{[\![}
\providecommand{\rBrack}{]\!]}

\title{An alternative approach to Fr\'echet derivatives}
\author{Shane Arora}
\author{Hazel Browne}
\author{Daniel Daners%
  \footnote{Email: saro0188@uni.sydney.edu.au (S.~Arora), %
    hbro4811@uni.sydney.edu.au (H.~Browne), %
    daniel.daners@sydney.edu.au (D.~Daners)}}%
\affil{School of Mathematics and Statistics, University of Sydney,
  NSW 2006, Australia}
\date{11 February 2019}

\makeatletter
\hypersetup{pdfauthor={Shane Arora, Hazel Browne, Daniel Daners}}
\hypersetup{pdftitle=\@title}
\makeatother

\numberwithin{equation}{section}
\numberwithin{figure}{section}

\theoremstyle{plain}
\newtheorem{theorem}{Theorem}[section]

\newtheorem{proposition}[theorem]{Proposition}
\newtheorem{corollary}[theorem]{Corollary}

\theoremstyle{definition}
\newtheorem{definition}[theorem]{Definition}

\newtheorem{example}[theorem]{Example}

\theoremstyle{remark}
\newtheorem{remark}[theorem]{Remark}

\DeclareMathOperator{\repart}{Re}

\begin{document}
\maketitle
\begin{abstract}
  We discuss an alternative approach to Fr\'echet derivatives on Banach
  spaces inspired by a characterisation of derivatives due to
  Carath\'eodory. The approach allows to reduce many questions of
  differentiability to a question of continuity. We demonstrate how that
  simplifies the theory of differentiation, including the rules of
  differentiation and the Schwarz Lemma on the symmetry of second order
  derivatives. We also provide a short proof of the differentiable
  dependence of fixed points in the Banach fixed point theorem.
\end{abstract}

\paragraph{MSC Classification (2010):} 46G05; 26B05

\section{Introduction}
\label{sec:intro}
The aim of this paper is to promote an alternative approach to Fr\'echet
derivatives of functions defined on open subsets of a real or complex
Banach space. The main feature is a simplification of many proofs by
reducing questions of differentiability to a question of continuity. The
approach is inspired by Carath\'eodory's characterisation of
differentiability of functions on the complex plane from
\cite{caratheodory:50:ft1} and its extension to vector valued functions
in \cite{acosta:94:fvc,cabrales:06:dfb}.

To motivate our approach let us start with the notion of tangent to the
graph of a function $f\colon J\to\mathbb R$, where $J\subseteq\mathbb R$
is an open interval. Given $x\in J$, the tangent to the graph of $f$ at
$(x,f(x))$ is by definition the limit of secants through the points
$(x,f(x))$ and $(y,f(y))$ as $y\to x$. The slope of that secant is given
by
\begin{equation}
  \label{eq:slope-single}
  \varphi_{x}(y):=\frac{f(y)-f(x)}{y-x}
\end{equation}
and we say that $f$ is differentiable at $x$ if $\varphi_{x}(y)$ has a
limit as $y\to x$. In other words, $\varphi_{x}$ has an extension from
$J\setminus\{x\}$ to $J$ that is continuous at $x$. Hence, $f$ is
differentiable at $x\in J$ if and only if there exists a function
$\varphi_{x}\colon J\to\mathbb R$, continuous at $y=x$, such that
\begin{equation}
  \label{eq:caratheodory-single}
  f(y)=f(x)+\varphi_{x}(y)(y-x)
\end{equation}
for all $y\in J$. By design, the derivative at $x$ is given by
$f'(x):=\varphi_{x}(x)$. We call $\varphi_{x}$ the \emph{slope function}
of $f$ at $x$.  The continuity of $\varphi_x$ at $x$ built into the
definition offers the biggest advantage over a traditional approach.

For functions between Banach spaces we can apply an idea similar to
\eqref{eq:caratheodory-single}.
\begin{definition}
  \label{def:caratheodory-diff}
  Let $E,F$ be real or complex Banach spaces and $U\subseteq E$
  open. Suppose that $f\colon U\to F$ and let $x\in U$. We say that $f$
  is \emph{Carath\'eodory differentiable} at $x$ if there exists a map
  $\Phi_x\colon U\to\mathcal L(E,F)$, continuous at $x$, such that
  \begin{equation}
    \label{eq:caratheodory-diff}
    f(y)=f(x)+\Phi_x(y)(y-x)
  \end{equation}
  for all $y\in U$. Here, $\mathcal L(E,F)$ is the space of bounded
  linear operators from $E$ to $F$ and continuity is with respect to the
  operator norm in $\mathcal L(E,F)$. We call $\Phi_x$ a \emph{slope
    function} of $f$ at $x$ and
  \begin{equation}
    \label{eq:deriv-at-a}
    Df(x):=\Phi_x(x)\in\mathcal L(E,F).
  \end{equation}
  the \emph{derivative} of $f$ at $x$.
\end{definition}
As $\Phi_x$ is continuous at $x$, it is a direct consequence of
\eqref{eq:caratheodory-diff} that $f$ is continuous at every point at
which it is differentiable.

We show in Section~\ref{sec:equivalence} that the above notion of
derivative is equivalent to the usual notion of \emph{Fr\'echet
  derivative}.  Adding to the exposition in
\cite{acosta:94:fvc,cabrales:06:dfb} we provide some geometric insight
and allow for any real or complex Banach space. As a demonstration of
the simplicity of the approach we then establish the standard rules of
differentiation in Section~\ref{sec:rules}.

To further support the case for our alternative approach to derivatives,
we provide short and conceptually simple proofs of two further
results. First, in Section~\ref{sec:symmetry}, we establish the Schwarz
Lemma about the symmetry of second order derivatives. Second, in
Section~\ref{sec:fixed-points}, we provide a simple proof of the
differentiable dependence of fixed points in the Banach Fixed Point
Theorem. That theorem can be applied directly to prove the inverse
function theorem or the differentiable dependence on parameters of
solutions to ordinary differential equations; see \cite{brooks:09:cmp}
for many such applications.

If $f$ is differentiable at every point $x\in U$, then it is convenient
to view the slope function as a function of two variables and write
\begin{equation}
  \label{eq:phi-two-var}
  \Phi(x,y):=\Phi_x(y),
\end{equation}
where $x$ is the point where we differentiate. By definition, the map
$y\mapsto\Phi(x,y)$ is continuous at $y=x$ and $Df(x)=\Phi(x,x)$. We
show that in general, the map $x\mapsto\Phi(x,y)$ cannot be expected to
be continuous at $x=y$, not even if $f$ is very smooth. In contrast to
that, we show that if $E$ is finite dimensional, then there always
exists a slope function that is \emph{separately} continuous on the
diagonal $x=y$ as a function of $x$ and $y$. Such examples are discussed
in Section~\ref{sec:discussion}.

While an arbitrary slope function can behave badly as a function of $x$
regardless of smoothness of $f$, we show that $f$ is continuously
differentiable at $x$ if and only if there exists a slope function that
is \emph{jointly} continuous on the diagonal as a function of both
variables. The proof, given in Section~\ref{sec:continuous-diff},
requires a mean value inequality which, unlike most references, we prove
for functions between complex Banach spaces.

We conclude this introduction by providing some historical comments. The
core idea goes back to the definition of derivative given by
Carath\'eodory in \cite{caratheodory:50:ft1}. However he does not really
make use of his definition, but instead reverts to a standard
approach. Others much later observed the usefulness. In the single
variable case, the most complete discussion appears in
\cite{kuhn:91:dac}.  In \cite[Section~III.6]{hairer:08:abh}, a
comparison of the definitions of derivatives due to Cauchy, Weierstrass
and Carath\'eodory is given, and Carath\'eodory's definition is used to
prove the standard rules of differentiation. The text
\cite{bartle:11:ira} uses Carath\'eodory's approach to prove some rules
of differentiation, but not beyond that. The approach is used quite
consistently in the calculus textbook \cite{martinez:98:cdc}.

The first time the definition seems to appear in the multi-variable case
is in \cite{botsko:85:dfs}. The most comprehensive exposition is given
in \cite{acosta:94:fvc}. There is a generalisation to functions on
Banach spaces in \cite{cabrales:06:dfb}, and \cite{pinzon:99:dcr}
focuses on the two variable case, providing comparisons with other
notions of differentiability. The definition also appears in
\cite[Section~IV.3]{hairer:08:abh}.

\section{Equivalence with Fr\'echet Derivatives}
\label{sec:equivalence}
Before we start our discussion of differentiability we need some
notation. The norm of $B\in\mathcal L(E,F)$ is the operator norm given
by
\begin{displaymath}
  \|B\|_{\mathcal L(E,F)}
  :=\sup_{x\in E\setminus\{0\}}\frac{\|Bx\|_F}{\|x\|_E}
  =\sup_{\|x\|_E\leq 1}\|Bx\|_F
  =\sup_{\|x\|_E=1}\|Bx\|_F;
\end{displaymath}
see for instance \cite[Section~II.1]{taylor:80:ifa}. A special case is
the dual space $E':=\mathcal L(E,\mathbb K)$ of $E$, where
$\mathbb K=\mathbb R$ if $E$ is a real Banach space and
$\mathbb K=\mathbb C$ if $E$ is complex. The dual norm
$\|{\cdot}\|_{E'}$ is just the operator norm in
$\mathcal L(E,\mathbb K)$. When no confusion is likely we denote the
norms on $E$ and $F$ simply by $\|{\cdot}\|$.

Let $f\colon U\to F$, where $U\subseteq E$ is open. The usual definition
of the derivative at $x\in U$ is the Fr\'echet derivative. The idea is
to find a linear operator $A\in\mathcal L(E,F)$ providing the \emph{best
  linear approximation} of $f$ near $x\in U$ in the sense that
\begin{equation}
  \label{eq:frechet-diff}
  \lim_{y\to x}\frac{f(y)-f(x)-A(y-x)}{\|y-x\|}=0
\end{equation}
in $F$. The map $A$ is called the derivative of $f$ at $x$ and is
denoted by $Df(x)$. The name goes back to Fr\'echet
\cite{frechet:12:ndt}, but Fr\'echet attributes the definition to Stolz
\cite{stolz:93:gdi}.

We now show that Fr\'echet's and Carath\'eodory's notions of derivatives
are equivalent. This is shown in \cite{acosta:94:fvc,cabrales:06:dfb},
but unlike these references we include a proof emphasising the geometric
significance of the constructions and allow for complex Banach spaces.

Assume that $f$ is Carath\'eodory differentiable in the sense of
Definition~\ref{def:caratheodory-diff} and set $A:=\Phi_x(x)$. Then,
\begin{multline*}
  \frac{\|f(y)-f(x)-A(y-x)\|}{\|y-x\|}
  =\frac{\bigl\|\bigl(\Phi_x(y)-\Phi_x(x)\bigr)(y-x)\bigr\|}{\|y-x\|}\\
  \leq\bigl\|\Phi_x(y)-\Phi_x(x)\bigr\|_{\mathcal L(E,F)}
  \frac{\|y-x\|}{\|y-x\|}
  =\bigl\|\Phi_x(y)-\Phi_x(x)\bigr\|_{\mathcal L(E,F)}.
\end{multline*}
Due to the continuity of $y\mapsto\Phi_x(y)$ at $x$ we know that
$\bigl\|\Phi_x(y)-\Phi_x(x)\bigr\|_{\mathcal L(E,F)}\to 0$ as $y\to x$
and hence \eqref{eq:frechet-diff} holds, showing that $f$ is Fr\'echet
differentiable.

Assuming that $f$ is Fr\'echet differentiable at $x$, we need to
construct a slope function $\Phi_x$ at $x$. Given $y\in U$ with
$y\neq x$, that slope function is uniquely defined in the direction of
$y-x$ by \eqref{eq:caratheodory-diff}, namely
$\Phi_x(y)(y-x)=f(y)-f(x)$. We then need to define $\Phi_x$ on a
subspace complementary to the line $\{t(y-x)\colon t\in\mathbb
K\}$. Such a complement is given by the kernel of a linear functional
$\ell(x,y)\in E'$ with $\langle\ell(x,y),y-x\rangle\neq 0$. For
$z\in\ker(\ell(x,y))$ we define $\Phi_x(y)z=Df(x)z$. That construction
is possible by the Hahn-Banach theorem which guarantees the existence of
a bounded linear functional $\ell(x,y)\in E'$ such that
$\langle\ell(x,y),y-x\rangle=\|y-x\|$ and $\|\ell(x,y)\|_{E'}=1$; see
\cite[Corollary~1.3]{brezis:11:fa}. Geometrically this means that, in
the direction of $\ker(\ell(x,y))$, the slope function $\Phi_x$ is
determined by the tangent of $f$ at $(x,f(x))$; see
Figure~\ref{fig:canonical-slope}. We can write
\begin{equation}
  \label{eq:canonical-slope-function}
  \Phi_x(y)z:=
  \begin{cases}
    \dfrac{f(y)-f(x)-Df(x)(y-x)}{\|y-x\|}\langle\ell(x,y),z\rangle
    +Df(x)z&\text{if $x\neq y$}\\
    Df(x)z&\text{if $x=y$}
  \end{cases}
\end{equation}
for all $z\in E$. This is a slope function since
$\langle\ell(x,y),y-x\rangle=\|y-x\|$ and so by construction
$f(y)=f(x)+\Phi_x(y)(y-x)$.  Moreover, since $\|\ell(x,y)\|_{E'}=1$
\begin{displaymath}
  \begin{split}
    \|\Phi_x(y)z-\Phi_x(x)z\|
    &=\dfrac{\|f(y)-f(x)-Df(x)(y-x)\|}{\|y-x\|}
    |\langle\ell(x,y),z\rangle|\\
    &\leq\dfrac{\|f(y)-f(x)-Df(x)(y-x)\|}{\|y-x\|} \|z\|
  \end{split}
\end{displaymath}
for all $z\in E$. By definition of the operator norm and since $f$ is
Fr\'echet differentiable,
\begin{displaymath}
  \|\Phi_x(y)-\Phi_x(x)\|_{\mathcal L(E,F)}
  \leq\dfrac{\|f(y)-f(x)-Df(x)(y-x)\|}{\|y-x\|}
  \xrightarrow{y\to x}0.
\end{displaymath}
Hence $f$ is Carath\'eodory differentiable. Note that if the dual norm
on $E'$ is strictly convex, then the functional $\ell(x,y)$ given by the
duality map is uniquely determined; see
\cite[Exercise~1.1]{brezis:11:fa}. For this reason we call
\eqref{eq:canonical-slope-function} the \emph{canonical slope function}.
We note that it is sufficient for $\ell(x,y)\in E'$ to have a bound
independent of $y$ in a neighbourhood of $x$ for the above arguments to
work.
\begin{figure}[ht]
  \centering
  \begin{tikzpicture}[scale=0.8]
    \coordinate (X) at (1.3,3.3);%
    \coordinate (V) at (-1.3,-0.9);%
    \coordinate (T) at (-2,0.65);%
    \begin{scope}
      \draw[clip] (-2,1.8)%
      .. controls (-1,3.3) and (0,3.8) .. (0.5,4)%
      .. controls (2,4.2) and (3,3.6) .. (4,3)%
      .. controls (2.9,2.5) and (2.2,1.8) .. (1.5,.7)%
      .. controls (0,1.8) and (-1.3,2) .. (-2,1.8)%
      -- cycle;%
      \path[ball color=lightgray] (1.2,0.8) circle (4cm);
    \end{scope}
    \path (1.5,0.7) node[right] {graph of $f$};%
    \draw[dashed] (X) -- ++(V);
    \begin{scope}[font=\footnotesize]
      \path[fill=black] (X) circle[radius=1.25pt] node[right=6pt]
      {$(x,f(x))$};%
      \draw[fill=white] ($(X)+(V)$) circle[radius=1.25pt] node[below
      right] {$(y,f(y)$};%
    \end{scope}
    \draw[gray] ($(X)+0.9*(T)-1.1*(V)$) -- ($(X)-1.1*(V)-1.2*(T)$)%
    -- ($(X)+2*(V)-1.2*(T)$) -- ($(X)+2*(V)+0.9*(T)$) -- cycle;%
    \begin{scope}[thick]
      \draw (X) -- ($(X)-1.4*(V)$) node[right] {secant};%
      \draw ($(X)+(V)$) -- ($(X)+2.1*(V)$); \draw ($(X)+(T)$) --
      ($(X)-1.3*(T)$) node[right] {tangent};%
    \end{scope}
    \draw[->] (-2.5,0) -- (-2.5,4.5) node[right] {$z$};%
    \draw[->] (-2.5,0) -- (1,-1.5) node[below] {$x_1$}; \draw[->]
    (-2.5,0) -- (0.5,.75) node[below] {$x_2$};
  \end{tikzpicture}
  \caption{Plane spanned by secant and tangent to define the canonical
    slope function.}
  \label{fig:canonical-slope}
\end{figure}
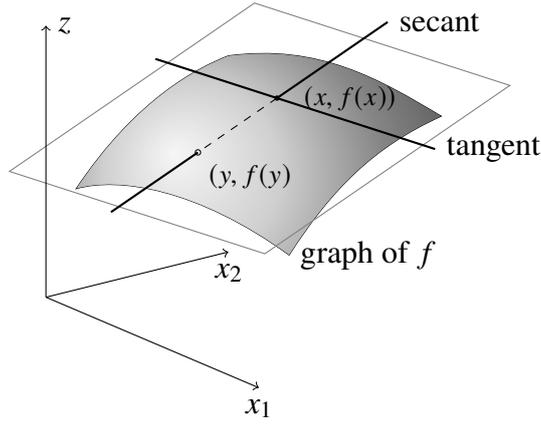
\par
We next look at some cases where it is possible to make a natural choice
for $\ell(x,y)$.
\begin{example}
  (a) If $E=H$ is a finite or infinite dimensional Hilbert space with
  inner product $\langle\cdot\,,\cdot\rangle_H$, then $\ell(x,y)$ is the
  orthogonal projection onto the subspace spanned by $y-x$, or more
  precisely the component in that direction. This is given by
  \begin{equation}
    \label{eq:l-Hilbert}
    \langle\ell(x,y),z\rangle
    :=\Bigl\langle\frac{y-x}{\|y-x\|_H},z\Bigr\rangle_H
  \end{equation}
  and illustrated in Figure~\ref{fig:canonical-slope}.  For a complex
  Hilbert space we take the inner product conjugate linear in the first
  argument.

  (b) If $E=L^p(\Omega)$ for some measure space $(\Omega,\mu)$ with
  $1<p<\infty$, then
  \begin{displaymath}
    \langle\ell(u,v),w\rangle
    :=\frac{1}{\|v-u\|_p^{p-1}}\int_\Omega|v-u|^{p-2}(\overline{v-u})w\,d\mu.
  \end{displaymath}
  Clearly $\langle\ell(u,v),v-u\rangle=\|v-u\|_p$ and by H\"older's
  inequality $|\langle\ell(u,v),w\rangle|\leq \|w\|_p$, so
  $\|\ell(u,v)\|_{(L^p)'}=1$.  In the Hilbert space case $p=2$ this
  coincides with \eqref{eq:l-Hilbert}.

  (c) If the norm $\|{\cdot}\|_E$ on $E$ is equivalent to a norm
  $\|{\cdot}\|_H$ induced by an inner product
  $\langle\cdot\,,\cdot\rangle_H$, then we can choose
  \begin{displaymath}
    \bigl\langle\ell(x,y),z\bigr\rangle
    :=\frac{\|y-x\|_E}{\|y-x\|_H}
    \Bigl\langle\frac{y-x}{\|y-x\|_H},z\Bigr\rangle_H.
  \end{displaymath}
  In particular, this is the case when working on any finite dimensional
  space such as $\mathbb R^n$ or $\mathbb C^n$, where every norm is
  equivalent to the Euclidean norm. We do not necessarily have
  $\|\ell(x,y)\|_{E'}=1$, but we still maintain the required uniform
  bound.
\end{example}
Given the non-uniqueness of complements of the space spanned by $y-x$
used to construct the slope function
\eqref{eq:canonical-slope-function}, it is clear that the slope function
cannot be unique unless $\dim(E)=1$. Also, the slope function does not
need to be of the form \eqref{eq:canonical-slope-function}. For examples
we refer to \cite[Section~2]{acosta:94:fvc} and to our more
comprehensive discussion of slope functions in
Section~\ref{sec:discussion}. However, the derivative is in fact
unique. We provide a proof, simpler than that given in
\cite[Section~2]{acosta:94:fvc}.
\begin{proposition}[Uniqueness of derivative]
  Let $E,F$ be Banach spaces, $U\subseteq E$ open and $f\colon U\to F$
  Carath\'eodory differentiable at $x\in U$. Then the derivative at $x$
  is unique.
\end{proposition}
\begin{proof}
  Let $\Phi_x\colon U\to\mathcal L(E,F)$ be an arbitrary slope
  function. Fix $z\in E$ and suppose that $t_0>0$ is small enough so
  that $x+tz\in U$ for all $t\in(0,t_0]$. This is possible since $U$ is
  open. By definition of $\Phi_x$ we have $f(x+tz)-f(x)=\Phi_x(x+tz)tz$
  for all $t\in(0,t_0]$ and so by the continuity of $\Phi_x$ at $x$,
  \begin{equation}
    \label{eq:derivative-unique}
    \lim_{t\to 0+}\frac{f(x+tz)-f(x)}{t}
    =\lim_{t\to 0+}\Phi_x(x+tz)z
    =\Phi_x(x)z.
  \end{equation}
  As the left hand side of \eqref{eq:derivative-unique} is independent
  of the particular slope function $\Phi_x$, it follows that
  $\Phi_x(x)z$ is uniquely determined by $f$, $x$ and $z$. As this is
  true for every $z\in E$ the derivative is unique.
\end{proof}
\begin{remark}
  If $f\colon\mathbb R^n\to\mathbb R^m$ (or
  $f\colon\mathbb C^n\to\mathbb C^m$), then the identity
  \eqref{eq:derivative-unique} also shows that the matrix representation
  of $Df(x)$ with respect to the standard basis is given by the Jacobian
  matrix. Indeed, if we choose $z=e_k$ to be the $k$-th standard basis
  vector of $\mathbb R^n$ (or $\mathbb C^n$), then the left hand side of
  \eqref{eq:derivative-unique} by definition is the partial derivative
  of $f$ with respect to $x_k$. Hence,
  \begin{displaymath}
    Df(x)e_k=\frac{\partial f}{\partial x_k}(x):=
    \begin{bmatrix}
      \frac{\partial f_1}{\partial x_k}(x)\\
      \vdots\\
      \frac{\partial f_m}{\partial x_k}(x)
    \end{bmatrix}
  \end{displaymath}
  for $k=1,\dots,n$, giving the $k$-th column of the Jacobian matrix.
\end{remark}
\section{The rules of differentiation}
\label{sec:rules}
The proofs of the standard rules of differentiation provide a convincing
case for the simplicity of Carath\'eodory's characterisation of
derivatives. The idea is always the same: through simple algebraic
manipulations we identify a slope function and exploit its continuity at
the point at which the derivative is taken. Unlike the traditional
approach, no ``$\varepsilon$-$\delta$'' or ``little $o$'' arguments are
needed, only clean and transparent arguments involving continuity
properties of the slope function and the function itself.
\begin{proposition}[Linearity]
  Suppose that $E,F$ are real or complex Banach spaces, that
  $U\subseteq E$ is open and that $f,g\colon U\to F$ are differentiable
  at $x\in U$. If $\lambda,\mu\in\mathbb R$ (or $\mathbb C$), then
  $D(\lambda f+\mu g)(x)=\lambda Df(x)+\mu Dg(x)$.
\end{proposition}
\begin{proof}
  Take slope functions $\Phi_x$ and $\Psi_x$ at $x$ for $f$ and
  $g$ respectively. Then
  \begin{displaymath}
    \lambda f(y)+\mu g(y)
    =\lambda f(x)+\mu g(x)
    +\bigl(\lambda\Phi_x(y)+\mu\Psi_x(y)\bigr)(y-x).
  \end{displaymath}
  Clearly, $\lambda\Phi_x(y)+\mu\Psi_x(y)\in\mathcal L(E,F)$ is
  continuous at $y=x$ and hence
  \begin{displaymath}
    D(\lambda f+\mu g)(x)
    =\lambda\Phi_x(x)+\mu\Psi_x(x)
    =\lambda Df(x)+\mu
    Dg(x)
  \end{displaymath}
  as claimed.
\end{proof}
We next prove the chain rule, which is a good example on how our approach
reduces questions about differentiability to questions of continuity by
identifying an appropriate slope function. Compare for instance with the
traditional proof of the chain rule in
\cite[Theorem~9.15]{rudin:76:pma}. The proof below is given in
\cite{acosta:94:fvc} for functions defined on Euclidean space, but
translates without change to real and complex Banach spaces.
\begin{theorem}[Chain rule]
  Suppose that $E,F,G$ are Banach spaces and that $U\subseteq E$ and
  $V\subseteq F$ are open sets. Assume that $g\colon U\to F$ is
  differentiable at $x\in U$ and that $g(x)\in V$. Further assume that
  $f\colon V\to G$ is differentiable at $g(x)$. Then $f\circ g$ is
  differentiable at $x$ and $D(f\circ g)(x)=Df\bigl(g(x)\bigr)Dg(x)$.
\end{theorem}
\begin{proof}
  Suppose that $\Phi\colon V\to \mathcal L(F,G)$ is a slope function of
  $f$ at $g(x)$ and that $\Psi\colon U\to\mathcal L(E,F)$ is a slope
  function of $g$ at $x$, that is,
  \begin{displaymath}
    \begin{aligned}
      f(z)&=f\bigl(g(x)\bigr)+\Phi(z)\bigl(z-g(x)\bigr)
      &&\text{for all $z\in V$}\\
      g(y)&=g(x)+\Psi(y)(y-x)&&\text{for all $y\in U$.}
    \end{aligned}
  \end{displaymath}
  In particular $f$ and $g$ are continuous at $g(x)$ and $x$,
  respectively. Using the two identities we can write
  \begin{displaymath}
    (f\circ g)(y)
    =f\bigl(g(x)\bigr)+\Phi\bigl(g(y)\bigr)\bigl(g(y)-g(x)\bigr)
    =(f\circ g)(x)+\Phi\bigl(g(y)\bigr)\Psi(y)(y-x).
  \end{displaymath}
  Hence, $y\mapsto\Lambda(y):=\Phi\bigl(g(y)\bigr)\Psi(y)$ is a slope
  function for $f\circ g$ at $x$. Using that the composition of
  continuous functions is continuous, $\Lambda$ is continuous at $y=x$
  and thus
  \begin{displaymath}
    D(f\circ g)(x)=\Lambda(x)=\Phi\bigl(g(x)\bigr)\Psi(x)
    =Df\bigl(g(x)\bigr)Dg(x)
  \end{displaymath}
  as claimed.
\end{proof}
We next prove a product rule. Products are not generally defined on
Banach spaces, but the main feature of products is that they are
bilinear. We let $E$, $F_1$ and $F_2$ be Banach spaces and
$U\subseteq E$ an open set. Let $G$ be another Banach space and assume
that $b\colon F_1\times F_2\to G$ is bounded and bilinear. Bounded means
that there exists $M>0$ such that
\begin{displaymath}
  \|b(y_1,y_2)\|_G\leq M\|y_1\|_{F_1}\|y_2\|_{F_2}
\end{displaymath}
for all $y_1\in F_1$ and $y_2\in F_2$. Given functions
$f_k\colon U\to F_k$, $k=1,2$, we consider $g\colon U\to G$ given by
\begin{displaymath}
  g(x):=b\bigl(f_1(x),f_2(x)\bigr)
\end{displaymath}
for all $x\in U$. The following proposition applies to pointwise
products of functions, the cross product, inner products and other
bilinear operations.
\begin{proposition}[Product rule]
  Let the above assumptions be satisfied and assume that $f_1,f_2$ are
  differentiable at $x\in U$. Then $g$ is differentiable with derivative
  given by
  \begin{equation}
    \label{eq:product-rule}
    Dg(x)z=b\bigl(Df_1(x)z,f_2(x))+b\bigl(f_1(x),Df_2(x)z\bigr)
  \end{equation}
  for all $z\in E$.
\end{proposition}
\begin{proof}
  Let $\Phi_1,\Phi_2$ be slope functions for $f_1$ and $f_2$ at $x$,
  respectively. Then, using that $b$ is bilinear, we obtain
  \begin{displaymath}
    \begin{split}
      g(y)&=b\bigl(f_1(y),f_2(y)\bigr)
      =b\bigl(f_1(x),f_2(y)\bigr)+b\bigl(\Phi_1(y)(y-x),f_2(y)\bigr)\\
      &=b\bigl(f_1(x),f_2(x)\bigr)+b\bigl(\Phi_1(y)(y-x),f_2(y)\bigr)
      +b\bigl(f_1(x),\Phi_2(y)(y-x)\bigr)\\
      &=g(x)+\Psi(y)(y-x),
    \end{split}
  \end{displaymath}
  where we have set
  \begin{displaymath}
    \Psi(y)z
    :=b\bigl(\Phi_1(y)z,f_2(y)\bigr)+b\bigl(f_1(x),\Phi_2(y)z\bigr)
  \end{displaymath}
  for all $z\in E$. As $b$ is bounded and bilinear we deduce that
  $\Psi(y)\in\mathcal L(E,G)$ is continuous at $y=x$, implying
  \eqref{eq:product-rule} since $\Phi_k(x)=Df_k(x)$ by definition.
\end{proof}
Another common rule of differentiation is the quotient rule, but like
the usual product rule it does not directly apply in Banach
spaces. Note, however, that the quotient rule is really a composition of
a function with inversion $t\mapsto 1/t=t^{-1}$ on $\mathbb R$ or
$\mathbb C$. Hence the natural generalisation of the quotient rule is
the derivative of the map $B\mapsto B^{-1}$ on the set of bounded
invertible linear operators.  It is known that this set is open in
$\mathcal L(E)$ and that the map $B\mapsto B^{-1}$ is continuous; see
for instance \cite[Theorem~IV.1.5]{taylor:80:ifa}. Based on this fact we
show that this map is also differentiable at every invertible
$A\in\mathcal L(E)$.

\begin{theorem}[Inversion]
  Let $A\in\mathcal L(E)$ be invertible. Then the map $f$ given by
  $f(B):=B^{-1}$ is differentiable at $A$, and for $Z\in\mathcal L(E)$,
  \begin{equation}
    \label{eq:quot-rule}
    Df(A)Z=-A^{-1}ZA^{-1}.
  \end{equation}
\end{theorem}
\begin{proof}
  If $A,B\in\mathcal L(E)$ are invertible, then
  \begin{displaymath}
    B^{-1}=A^{-1}-A^{-1}+B^{-1}
    =A^{-1}-A^{-1}(B-A)B^{-1}
    =A^{-1}+\Phi(A,B)(B-A),
  \end{displaymath}
  where $\Phi(A,B)Z:=-A^{-1}ZB^{-1}$ for all $Z\in\mathcal L(E)$.  Then
  $\Phi(A,B)\in\mathcal L\bigl(\mathcal L(E)\bigr)$ and
  \begin{displaymath}
    \|\Phi(A,B)Z-\Phi(A,A)Z\|
    =\|A^{-1}Z(A^{-1}-B^{-1})\|
    \leq\|A^{-1}\|\|A^{-1}-B^{-1}\|\|Z\|.
  \end{displaymath}
  Here, $\|{\cdot}\|$ is the norm in $\mathcal L(E)$. By definition of
  the operator norm, continuity of inversion and since the set of
  invertible operators is open,
  \begin{displaymath}
    \|\Phi(A,B)-\Phi(A,A)\|_{\mathcal L(\mathcal L(E))}
    \leq\|A^{-1}\|\|A^{-1}-B^{-1}\|\to 0
  \end{displaymath}
  as $B\to A$ in $\mathcal L(E)$. Hence $\Phi$ is a slope function for
  $f$ and $\Phi(A,B)\to\Phi(A,A)$ in $\mathcal L(\mathcal L(E))$,
  proving \eqref{eq:quot-rule}.
\end{proof}
Finally we look at functions on a product space and partial
derivatives.
\begin{proposition}[Partial derivatives]
  \label{prop:partial-derivatives}
  Let $E_1,E_2$ and $F$ be Banach spaces and let
  $U\subseteq E_1\times E_2$ be open. Assume that $f\colon U\to F$ is
  differentiable at $x=(x_1,x_2)\in U$ with slope function $\Phi$. For
  $z_1\in E_1$ we define the partial slope function $\Phi_1$ by
  \begin{equation}
    \label{eq:partial-slope-function}
    \Phi_1(x,y_1)z_1:=\Phi\bigl(x,(y_1,x_2)\bigr)(z_1,0)
  \end{equation}
  Then the function $y_1\mapsto f(y_1,x_2)$ defined on
  $U_{x_2}:=\{y_1\in E_1\colon (y_1,x_2)\in U\}$ is differentiable with
  slope function $\Phi_1(x,\cdot)\colon U_{x_2}\to\mathcal L(E_1)$ and
  derivative given by $D_1f(x_1,x_2)z_1=Df(x)(z_1,0)$ for all
  $z_1\in E_1$.
\end{proposition}
\begin{proof}
  By definition of a slope function and
  \eqref{eq:partial-slope-function},
  \begin{displaymath}
    f(y_1,x_2)=f(x)+\Phi\bigl(x,(y_1,x_2)\bigr)(y_1-x_1,0)
    =f(x)+\Phi_1\bigl(x,y_1\bigr)(y_1-x_1).
  \end{displaymath}
  From properties of $\Phi$ we have that $\Phi_1(x,y_1)\to\Phi_1(x,x_1)$
  in $\mathcal L(E_1,F)$ as $y_1\to x_1$. Hence, $y_1\mapsto f(y_1,x_2)$
  is differentiable at $x_1$ with slope function $\Phi_1(x,\cdot)$ at
  $x_1$.
\end{proof}

Note that the slope function of $y_1\mapsto f(y_1,x_2)$ depends on
$x_2$. For that reason we have kept $x=(x_1,x_2)$ as the first argument
of $\Phi_1$ and not just $x_1$.  As usual, we sometimes write
$D_{x_1}f(x_1,x_2)$ or $D_1f(x_1,x_2)$ for the partial derivative. In a
similar fashion we obtain the partial derivative with respect to
$x_2$. A similar approach works for products of more than two spaces.

\section{Characterisation of continuous differentiability}
\label{sec:continuous-diff}
Assume that $U\subseteq E$ is open and that $f\colon U\to F$ is
differentiable. For every slope function $\Phi(x,y)$ we require the
continuity of $y\mapsto\Phi(x,y)$ at $x$ by definition. We do not say
anything about continuity as a function of $x$, let alone joint
continuity as a function of $(x,y)$. It turns out that continuous
differentiability can be characterised by means of such a joint
continuity property. Such a characterisation appears in
\cite[Section~5]{acosta:94:fvc}, but apart from a generalisation to the
Banach space case we also provide details on how exactly the mean value
theorem is used.
\begin{theorem}
  \label{thm:continuous-diff}
  Suppose that $E,F$ are Banach spaces, that $U\subseteq E$ is open and
  that $f\colon U\to F$ is differentiable. Then $Df$ is continuous at
  $x_0$ if and only if there exists a slope function
  $\Phi(\cdot\,,\cdot)$ for $f$ that is (jointly) continuous at
  $(x_0,x_0)$ as a function of $(x,y)$. In that case, the canonical
  slope function given by \eqref{eq:canonical-slope-function} is jointly
  continuous at $(x_0,x_0)$.
\end{theorem}
It is tempting to believe that every slope function has the above joint
continuity property if $Df$ is continuous at $x_0$. However, as we show
in Section~\ref{sec:discussion}, one can always construct a slope
function that is not even separately continuous. This is not bad because
in practice we only need to know that a jointly continuous slope
function exists. Note also that we make no claim on the continuity of
$\Phi$ at points other than $(x_0,x_0)$.

The main tool to prove the above theorem is a mean value inequality. To
simplify the statement we denote the line segment connecting $x$ and $y$
in $E$ by
\begin{displaymath}
  \lBrack x,y\rBrack:=\bigl\{x+t(y-x)\colon t\in[0,1]\bigr\}.
\end{displaymath}
The idea is taken from \cite[Theorem~5.19]{rudin:76:pma}, but instead of
inner products in $\mathbb R^n$ we use duality in Banach spaces. We also
deal with the case of complex Banach spaces.
\begin{theorem}[Mean value inequality]
  \label{thm:mvt-ineq}
  Assume that $E,F$ are Banach spaces, that $U\subseteq E$ is open and
  that $f\colon U\to F$ is differentiable. Let $A\in\mathcal L(E,F)$ and
  let $x,y\in U$ be distinct points such that
  $\lBrack x,y\rBrack\subseteq U$. Then there exists
  $c\in\lBrack x,y\rBrack$, $c\neq x,y$, such that
  \begin{equation}
    \label{eq:mvt-ineq}
    \|f(y)-f(x)-A(y-x)\|
    \leq\|Df(c)(y-x)-A(y-x)\|.
  \end{equation}
\end{theorem}
\begin{proof}
  By the Hahn-Banach theorem there exists $\varphi\in F'$ with
  $\|\varphi\|_{F'}=1$ such that
  \begin{equation}
    \label{eq:hb}
    \langle\varphi,f(y)-f(x)-A(y-x)\rangle
    =\|f(y)-f(x)-A(y-x)\|_F;
  \end{equation}
  see \cite[Corollary~1.3]{brezis:11:fa}.  We next define the function
  $g\colon [0,1]\to\mathbb C$ by
  \begin{equation}
    \label{eq:hb-g}
    g(t)
    :=\bigl\langle\varphi,f\bigl(x+t(y-x)\bigr)-f(x)-tA(y-x)\bigr\rangle.
  \end{equation}
  It is well defined since $\lBrack x,y\rBrack\subseteq U$ by
  assumption.  It is real valued if $E,F$ are real Banach spaces. To
  allow for complex Banach spaces we define the function
  $H\colon [0,1]\to\mathbb R$ by
  \begin{displaymath}
    H(t)
    =\repart\bigl(\overline{g(1)}g(t)\bigr)
  \end{displaymath}
  for all $t\in[0,1]$. We note that a complex valued function of
  $t\in\mathbb R$ is differentiable if and only if its real and
  imaginary parts are differentiable. As $g(0)=0$ and hence $H(0)=0$, by
  the classical mean value theorem there exists $t_0\in(0,1)$ such that
  \begin{displaymath}
    |g(1)|^2
    =H(1)-H(0)
    =H'(t_0)
    =\repart\bigl(\overline{g(1)}g'(t_0)\bigr)
    \leq |g(1)||g'(t_0)|.
  \end{displaymath}
  Hence,
  \begin{math}
    |g(1)|\leq|g'(t_0)|.
  \end{math}
  Using \eqref{eq:hb}, \eqref{eq:hb-g} and the chain rule we deduce that
  \begin{displaymath}
    \begin{split}
      \|f(y)-f(x)-A(y-x)\|
      &=|g(1)|\leq |g'(t_0)|\\
      &=\bigl|\bigl\langle\varphi,Df(x+t_0(y-x))(y-x)
      -A(y-x)\bigr\rangle\bigr|\\
      &\leq\|Df(x+t_0(y-x))(y-x)-A(y-x)\|.
    \end{split}
  \end{displaymath}
  In the last step we used that $\|\varphi\|_{F'}=1$. To complete the
  proof of \eqref{eq:mvt-ineq} we finally set $c:=x+t_0(y-x)$. Clearly
  $c\in\lBrack x,y\rBrack$, $c\neq x,y$, since $t_0\in(0,1)$.
\end{proof}
From the above mean value inequality we can derive an inequality
involving the special slope function \eqref{eq:canonical-slope-function}
\begin{corollary}
  \label{cor:mvt-ineq}
  Suppose that the assumptions of Theorem~\ref{thm:mvt-ineq} are
  satisfied, and that $\Phi(x,y)$ is a slope function of $f$ of the form
  \eqref{eq:canonical-slope-function}. If $x,y\in U$ are distinct points
  such that $\lBrack x,y\rBrack\subseteq U$, then there exists
  $c\in\lBrack x,y\rBrack$, $c\neq x,y$, such that
  \begin{equation}
    \label{eq:mvt-ineq-general}
    \|\Phi(x,y)-A\|_{\mathcal L(E,F)}
    \leq\|Df(c)-Df(x)\|_{\mathcal L(E,F)}+\|Df(x)-A\|_{\mathcal L(E,F)}.
  \end{equation}
\end{corollary}
\begin{proof}
  We start by noting that, for all $z\in E$,
  \begin{align*}
    \|\Phi(x,y)z&-Az\|
      =\Bigl\|\frac{f(y)-f(x)-Df(x)(y-x)}{\|y-x\|}\langle\ell(x,y),z\rangle
      +Df(x)z-Az\Bigr\|\\
    &\leq\Bigl\|\frac{f(y)-f(x)-Df(x)(y-x)}{\|y-x\|}\Bigr\|
      \|\ell(x,y)\|_{E'}\|z\|
      +\|Df(x)-A\|_{\mathcal L(E,F)}\|z\|\\
    &=\Bigl\|\frac{f(y)-f(x)-Df(x)(y-x)}{\|y-x\|}\Bigr\|\|z\|
      +\|Df(x)-A\|_{\mathcal L(E,F)}\|z\|,
  \end{align*}
  where we used that $\|\ell(x,y)\|_{E'}=1$. Hence by definition of the
  operator norm,
  \begin{displaymath}
    \|\Phi(x,y)-A\|_{\mathcal L(E,F)}
    \leq\Bigl\|\frac{f(y)-f(x)-Df(x)(y-x)}{\|y-x\|}\Bigr\|
    +\|Df(x)-A\|_{\mathcal L(E,F)}.
  \end{displaymath}
  Applying Theorem~\ref{thm:mvt-ineq}, there exists
  $c\in\lBrack x,y\rBrack$ with
  \begin{multline*}
    \Bigl\|\frac{f(y)-f(x)-Df(x)(y-x)}{\|y-x\|}\Bigr\|
    \leq\frac{1}{\|y-x\|}\bigl\|Df(c)(y-x)-Df(x)(y-x)\bigr\|\\
    \leq\bigl\|Df(c)-Df(x)\bigr\|_{\mathcal L(E,F)}
    \frac{\|y-x\|}{\|y-x\|}
    =\bigl\|Df(c)-Df(x)\bigr\|_{\mathcal L(E,F)}.
  \end{multline*}
  Combining the above, \eqref{eq:mvt-ineq-general} follows.
\end{proof}
\begin{remark}
  As seen from the proof of Theorem~\ref{thm:mvt-ineq} and
  Corollary~\ref{cor:mvt-ineq}, it is sufficient to assume that $f$ be
  continuous at the endpoints of $\lBrack x,y\rBrack$ and differentiable
  inside.
\end{remark}

Now we are in a position to prove Theorem~\ref{thm:continuous-diff}.
\begin{proof}[Proof of Theorem~\ref{thm:continuous-diff}]
  First assume that there exists a slope function $\Phi$ that is
  continuous at $(x_0,x_0)$. Then in particular the function
  $x\mapsto\Phi(x,x)=Df(x)$ is continuous at $x_0$, that is, $Df$ is
  continuous at $x_0$.

  Assume now that $Df$ is continuous at $x_0$. We choose the slope
  function $\Phi(x,y)$ of $f$ given by
  \eqref{eq:canonical-slope-function}.  As $U$ is open we can find $r>0$
  such that $B(x_0,r)\subseteq U$. If we fix $x,y\in B(x_0,r)$, then,
  applying \eqref{eq:mvt-ineq-general} with $A=Df(x_0)$, there exists
  $c_{x,y}\in\lBrack x,y\rBrack$ with
  \begin{equation}
    \label{eq:diag-estimate}
    \|\Phi(x,y)-Df(x_0)\|_{\mathcal L(E,F)}
    \leq\bigl\|Df(c_{x,y})-Df(x)\bigr\|_{\mathcal L(E,F)}
    +\bigl\|Df(x)-Df(x_0)\|_{\mathcal L(E,F)}.
  \end{equation}
  As $c_{x,y}$ is a convex combination of $x$ and $y$, it follows that
  $c_{x,y}\in B(x_0,r)$ and that $c_{x,y}\to x_0$ as
  $(x,y)\to (x_0,x_0)$.  By the continuity of $Df$ at $x_0$, we deduce
  from \eqref{eq:diag-estimate} that
  \begin{displaymath}
    \lim_{(x,y)\to (x_0,x_0)}\bigl\|\Phi(x,y)-Df(x_0)\|_{\mathcal
      L(E,F)}=0,
  \end{displaymath}
  proving the joint continuity of $\Phi$ at $(x_0,x_0)$.
\end{proof}

\section{The symmetry of second order derivatives}
\label{sec:symmetry}
If $f\colon U\to E$ is differentiable, then it makes sense to consider
the second order derivative. As $Df\colon U\to\mathcal L(E,F)$, the
second order derivative $D^2f(x)$ is a linear operator from $E$ into
$\mathcal L(E,F)$, that is,
$D^2f(x)\in\mathcal L\bigl(E,\mathcal L(E,F)\bigr)$. As commonly done,
we identify $\mathcal L\bigl(E,\mathcal L(E,F)\bigr)$ with the space
$\mathcal L^2(E\times E;F)$ of bounded bilinear maps from $E\times E$ to
$F$; see for instance \cite[Theorem~4.3]{amann:08:a2}. With that
identification $D^2f(x)\in\mathcal L^2(E\times E;F)$. We use the theory
developed so far to provide a simple proof of the well known fact that
$D^2f(x)$ is symmetric, named after Schwarz, Young or Clairaut depending
on local tradition. Most references provide a proof if the second order
derivative is continuous. We only assume that it exists at one point.

\begin{theorem}[Symmetry of second order derivatives]
\label{thm:schwarz}
Assume that $f\colon U\to F$ is such that $D^2f(x)$ exists at the point
$x\in U$. Then $D^2f(x)$ is symmetric, that is,
$D^2f(x)[u,v]=D^2f(x)[v,u]$ for all $u,v\in E$.
\end{theorem}
\begin{proof}
  We first note that for $D^2f(x)$ to exist, $f$ needs to be
  differentiable in a neighbourhood of $x$. We fix $u,v\in E$. As $f$ is
  differentiable in a neighbourhood of $x$, for fixed $s>0$ small
  enough, the function $g\colon[0,s]\to F$ given by
  \begin{equation}
    \label{eq:f-difference}
    g(t):=f(x+su+tv)-f(x+tv)
  \end{equation}
  is well defined and differentiable. Thus the mean value inequality
  from Theorem~\ref{thm:mvt-ineq} implies the existence of
  $\theta\in(0,1)$ such that
  \begin{equation}
    \label{eq:sd-diff}
    \|g(s)-g(0)-s^2D^2f(x)[u,v]\|
    \leq\|g'(\theta s)s-s^2D^2f(x)[u,v]\|,
  \end{equation}
  where we have set $At:=tsD^2f(x)[u,v]$ for the linear map
  $A\colon\mathbb R\to F$.  As $Df$ is differentiable at $x$ there
  exists a slope function $\Phi\colon U\to\mathcal L^2(E\times E;F)$ for
  $Df$ at $x$. Using the chain rule to compute $g'$ we see that
  \begin{displaymath}
    \begin{split}
      g'(\theta s)
      &=Df(x+su+\theta sv)v-Df(x+\theta sv)v\\
      &=\bigl(Df(x+su+\theta sv)-Df(x)\bigr)v
      -\bigl(Df(x+\theta sv)-Df(x)\bigr)v\\
      &=\Phi(x+su+\theta sv)[su+\theta sv,v]
      -\Phi(x+\theta sv)[\theta sv,v]\\
      &=s\bigl(\Phi(x+su+\theta sv)-\Phi(x+\theta sv)\bigr)[\theta v,v]
      +s\Phi(x+su+\theta sv)[u,v].
    \end{split}
  \end{displaymath}
  Combining the above identity with \eqref{eq:sd-diff} and using that
  $\theta\in(0,1)$, we arrive at
  \begin{equation}
    \label{eq:mvt-est}
    \begin{split}
      \Bigl\|\frac{g(s)-g(0)}{s^2}&-D^2f(x)[u,v]\Bigr\|_F
      \leq\Bigl\|\frac{g'(\theta s)}{s}-D^2f(x)[u,v]\Bigr\|_F\\
      &\leq\bigl\|\Phi(x+su+\theta sv)-\Phi(x+\theta sv)\bigr\|_{\mathcal
        L^2(E\times E;F)}\|v\|_E^2\\
      &\qquad+\bigl\|\Phi(x+su+\theta sv)-D^2f(x)\bigr\|_{\mathcal
        L^2(E\times E;F)}\|u\|_E\|v\|_E.
    \end{split}
  \end{equation}
  By definition of differentiability, $\Phi$ is continuous at $x$ and
  hence
  \begin{displaymath}
    \lim_{s\to 0+}\Phi(x+su+\theta sv)
    =\lim_{s\to 0+}\Phi(x+\theta sv)
    =D^2f(x)
  \end{displaymath}
  in $\mathcal L^2(E\times E;F)$. Hence, the right hand side of
  \eqref{eq:mvt-est} goes to zero as $s\to 0+$, that is,
  \begin{displaymath}
    \lim_{s\to 0+}\frac{g(s)-g(0)}{s^2}=D^2f(x)[u,v].
  \end{displaymath}
  Looking at the definition of $g$ given in \eqref{eq:f-difference} we
  see that $g(s)-g(0)$ is symmetric as a function of $(u,v)$, so by
  interchanging the roles of $u$ and $v$ we also have
  \begin{displaymath}
    \lim_{s\to 0+}\frac{g(s)-g(0)}{s^2}=D^2f(x)[v,u],
  \end{displaymath}
  proving that $D^2f(x)[u,v]=D^2f(x)[v,u]$.
\end{proof}
\begin{remark}
  By an induction argument, the above theorem implies the symmetry of
  all higher order derivatives. The induction argument used in
  \cite[Corollary~VII.4.7]{amann:08:a2} can be adapted for that
  purpose. In the case of a function $f\colon\mathbb R^n\to\mathbb R$,
  symmetry means that the Hessian matrix is symmetric, and more
  generally that partial derivatives can be taken in any order to yield
  the same result.
\end{remark}

\section{Further discussion of slope functions}
\label{sec:discussion}
In this section we provide a further discussion of slope functions. In
particular we discuss joint and separate continuity, symmetry, and
derivatives of Lipschitz functions.

\paragraph{Joint and separate continuity.}
If $g\colon\mathbb R\to\mathbb R$ is differentiable, then the slope
function is uniquely determined and given by
\begin{equation}
  \label{eq:slope-1v}
  \varphi(s,t):=
  \begin{cases}
    \dfrac{g(t)-g(s)}{t-s}&\text{if $t\neq s$,}\\
    g'(s)&\text{if $t=s$.}
  \end{cases}
\end{equation}
Clearly $\varphi(s,t)=\varphi(t,s)$ and hence $\varphi$ is separately
continuous at $(s,s)$, that is, $t\mapsto \varphi(s,t)$ is continuous at
$s$ and $t\mapsto\varphi(t,s)$ is continuous at $s$. We show that this
is not necessarily the case for functions of two or more variables.
\begin{example}
  \label{ex:sep-contl-1}
  For $s\in\mathbb R$ define $g(s):=s^2\cos(1/s)$ if $s\neq 0$ and
  $g(0):=0$.  We can define a function of two variables by setting
  \begin{displaymath}
    f(x):=g(x_1)
  \end{displaymath}
  for all $x=(x_1,x_2)\in\mathbb R^2$.  If $x_1=0$, then the canonical
  slope function \eqref{eq:canonical-slope-function} is the $1\times 2$
  matrix given by
  \begin{displaymath}
    \Phi(x,y)=\frac{f(y)}{\|y-x\|^2}
    \begin{bmatrix}
      y_1&y_2-x_2
    \end{bmatrix}.
  \end{displaymath}
  for all $y\neq x$.  If $x_1\neq 0$, then it is given by
  \begin{displaymath}
    \begin{split}
      \Phi(x,y)&=\frac{f(y)-f(x)-Df(x)(y-x)}{\|y-x\|^2}
      \begin{bmatrix}
        y_1-x_1&y_2-x_2
      \end{bmatrix}
      +Df(x)\\
      &=\frac{y_1-x_1}{\|y-x\|^2}\bigl(\varphi(x_1,y_1)-g'(x_1)\bigr)
      \begin{bmatrix}
        y_1-x_1&y_2-x_2
      \end{bmatrix}
      +
      \begin{bmatrix}
        g'(x_1)&0
      \end{bmatrix},
    \end{split}
  \end{displaymath}
  where $\varphi$ is the slope function of $g$ given by
  \eqref{eq:slope-1v}.  Obviously $\Phi(0,x)\neq\Phi(x,0)$ which is not
  surprising given the geometric interpretation of slope functions from
  Section~\ref{sec:equivalence}. What is more interesting is that $\Phi$
  is not separately continuous at $(0,0)$. In particular,
  $\lim_{x\to 0}\Phi(x,0)$ does not exist. Indeed, since
  $|g'(x_1)|\leq 2$ for $|x_1|\leq 1$ and $\varphi(x_1,0)\to g'(0)=0$ as
  $x_1\to 0$ it follows that
  \begin{displaymath}
    \lim_{x\to 0}\frac{x_1}{\|x\|^2}\bigl(\varphi(x_1,0)-g'(x_1)\bigr)
      \begin{bmatrix}
        x_1&x_2
      \end{bmatrix}
      =0.
  \end{displaymath}
  However, the second term $[g'(x_1)\;0]$ does not converge as
  $x_1\to 0$.

  The above example also shows that the canonical slope function is not
  always the best one to use. Here, there is a much simpler one with
  much better properties, namely
  \begin{displaymath}
    \Psi(x,y):=
    \begin{bmatrix}
      \varphi(x_1,y_1)&0
    \end{bmatrix}.
  \end{displaymath}
  Inheriting the properties of $\varphi$, it follows that $\Psi$ is
  separately continuous and symmetric.
\end{example}
At every point $(0,x_2)$, the function $f$ in the above example is not
continuously differentiable. We now show that separate continuity of the
slope function can fail regardless of how smooth the function is. This
makes it clear that Theorem~\ref{thm:continuous-diff} is optimal in the
sense that it can only ever assert the existence of a jointly continuous
slope function, but nothing can be said about an arbitrary slope
function.
\begin{example}
  \label{ex:sep-contl-2}
  Consider the zero function $f(x):=0$ for all $x\in\mathbb R^2$, whose
  derivative is given by $Df(x)=[0\;0]$ for all $x\in\mathbb
  R^2$. Suppose that $g\colon\mathbb R^2\times\mathbb R^2\to\mathbb R$
  is such that
  \begin{equation}
    \label{eq:g-y-cont}
    \lim_{y\to x}g(x,y)=0.
  \end{equation}
  Then the linear operators
  $\Phi(x,y)\in\mathcal L(\mathbb R^2,\mathbb R)$ given by
  \begin{displaymath}
    \Phi(x,y):=g(x,y)
    \begin{bmatrix}
      -\dfrac{y_2-x_2}{\|y-x\|}&\dfrac{y_1-x_1}{\|y-x\|}
    \end{bmatrix}
  \end{displaymath}
  if $x\neq y$ and $\Phi(x,x):=[0\;0]$ defines a slope function for $f$
  at $x$. Indeed, note that $\Phi(x,y)(y-x)=0$ and that
  $\|\Phi(x,y)\|\leq|g(x,y)|$ for all $x,y\in\mathbb R^2$. Therefore, by
  \eqref{eq:g-y-cont}, for every $x\in\mathbb R^2$ we have
  $\Phi(x,y)\to[0\;0]$ as $y\to x$. We choose $g$ to be given by
  \begin{displaymath}
    g(x,y):=
    \begin{cases}
      1&\text{if $y=0$ and $x\neq 0$,}\\
      0&\text{otherwise.}
    \end{cases}
  \end{displaymath}
  Then \eqref{eq:g-y-cont} holds for all $x\in\mathbb R^2$, but
  $g(x,0)=1\to 1\neq 0=g(0,0)$ as $x\to 0$. In particular,
  \begin{displaymath}
    \lim_{y\to 0}\Phi(0,y)=
    \begin{bmatrix}
      0&0
    \end{bmatrix}
    \qquad\text{but}\quad
    \lim_{x\to 0}\Phi(x,0)\quad\text{does not exist.}
  \end{displaymath}
  This means that $\Phi$ is not separately continuous as a function of
  $x$ and $y$ at $(0,0)$ even though $f$ is as smooth as we like. Given
  an arbitrary smooth function from $\mathbb R^2$ to $\mathbb R$ we can
  always add $\Phi$ to the corresponding slope function and get a badly
  behaved one. Likewise, we can do that at any point in the domain by
  translation.

  The example can be modified to work on any Banach space $E$ by looking
  at a pair of non-trivial complemented subspaces $E=E_1\oplus E_2$ and
  choosing $x\in E_1$ and $y\in E_2$.
\end{example}

In contrast to the above examples we show that at least in finite
dimensions, for any differentiable function (not necessarily
continuously differentiable) one can always construct a separately
continuous and symmetric slope function.

\paragraph{Symmetry and separate continuity of the slope function.}
We know that the slope function $\varphi$ of a differentiable function
of one variable is symmetric, that is, $\varphi(x,y)=\varphi(y,x)$. We
also know from previous discussions and Example~\ref{ex:sep-contl-1}
that this is not necessarily the case for any given slope function
$\Phi$ of a function of several variables. If $\Phi$ is separately
continuous, then the symmetric part
\begin{equation}
  \label{eq:symmetric-part}
  \Psi(x,y):=\frac{1}{2}\bigl(\Phi(x,y)+\Phi(y,x)\bigr)
\end{equation}
is a slope function. Hence, there is a symmetric slope function if and
only if there exists a separately continuous slope function. If $f$ is
continuously differentiable, then, by Theorem~\ref{thm:continuous-diff},
we have such a slope function. We could ask whether it is possible to
construct a separately continuous slope function for a function that is
just differentiable. It turns out that this is the case for a function
of finitely many variables.

Given a differentiable function $f\colon U\to\mathbb R^m$,
$U\subseteq\mathbb R^n$ open, we now construct a separately continuous
slope function. The construction comes closest to the definition of a
derivative for a function of one variable as a limit of secants. The
idea is to consider a secant plane and pass to the limit to obtain the
tangent plane.

For each pair of points $x,y$ set $v_1=\dfrac{y-x}{\|y-x\|}$ and choose
vectors $v_k$, $k=2,\dots,n$, so that $(v_1,v_2,\dots,v_{n})$ forms an
orthonormal basis of $\mathbb R^n$. In what follows we should keep in
mind that the vectors $v_k$ depend on the direction of $y-x$, but in
order to keep the notation simple we do not indicate that dependence
explicitly. We now define a linear operator
$\Phi(x,y)\in\mathcal L(\mathbb R^n,\mathbb R^m)$ by defining it on the
basis $(v_1,\dots,v_n)$ by
\begin{equation}
  \label{eq:1}
  \Phi(x,y)v_k:=\frac{f\bigl(x+\|y-x\|v_k\bigr)-f(x)}{\|y-x\|}
\end{equation}
for $k=1,\dots,n$. We claim that $\Phi$ is a slope function. By
\eqref{eq:1} and the definition of $v_1$,
\begin{displaymath}
  \Phi(x,y)(y-x)
  =\|y-x\|\Phi(x,y)v_1
  =f(y)-f(x).
\end{displaymath}
To check continuity at $x$ as a function of $y$, write $z\in \mathbb R^n$
in the form $z=\sum_{k=1}^n\alpha_kv_k$, where
$\alpha_k:=\langle v_k,z\rangle$.  As the basis $(v_1,\dots,v_n)$ is
orthonormal we have $\|z\|^2=\sum_{k=1}^n|\alpha_k|^2$ and thus, by the
Cauchy Schwarz inequality,
\begin{align*}
  \|\Phi(x,y)z
  &-Df(x)z\|
  =\Bigl\|\sum_{k=1}^n\alpha_k
    \frac{f\bigl(x+\|y-x\|v_k\bigr)-f(x)
    -Df(x)\|y-x\|v_k}{\|y-x\|}\Bigr\|\\
  &\leq\|z\|\sqrt{\sum_{k=1}^n
    \left(\frac{\bigl\|f\bigl(x+\|y-x\|v_k\bigr)-f(x)
    -Df(x)\|y-x\|v_k\bigr\|}{\|y-x\|}\right)^2}
  \to 0
\end{align*}
as $y\to x$ by differentiability of $f$ at $x$.  We next show that
$\Phi(x,y)$ is continuous as a function of $x$ as $x\to y$. The trick is
to rewrite $\Phi(x,y)$ with respect to the basis
\begin{displaymath}
  (w_1,\dots,w_n):=(-v_1,v_2-v_1,\dots,v_{n}-v_1).
\end{displaymath}
As $x=y-(y-x)=y-\|x-y\|v_1=y+\|x-y\|w_1$ we conclude that for
$k=2,\dots,n$,
\begin{equation}
  \label{eq:3}
  x+\|y-x\|v_k=y+\|x-y\|(v_k-v_1)=y+\|x-y\|w_k.
\end{equation}
Hence, by using \eqref{eq:1}, we obtain for $k=2,\dots,n$,
\begin{equation}
  \label{eq:4}
  \Phi(x,y)w_k=\Phi(x,y)v_k-\Phi(x,y)v_1
  =\frac{f\bigl(y+\|x-y\|w_k\bigr)-f(y)}{\|x-y\|}.
\end{equation}
Note that the final formula also applies to $k=1$.  Expressing $z$ in
terms of the basis $(w_1,\dots,w_n)$, it turns out that
\begin{displaymath}
  z=\sum_{k=1}^n\alpha_kv_k
  =\sum_{k=2}^n\alpha_kw_k-\Bigl(\sum_{k=1}^n\alpha_k\Bigr)w_1.
\end{displaymath}
If we set $\beta_1:=-\sum_{k=1}^n\alpha_k$ and $\beta_k:=\alpha_k$ for
$k=2,\dots,n$, we see that
\begin{displaymath}
  \sum_{k=1}^n|\beta_k|^2
  \leq \sum_{k=2}^n|\alpha_k|^2+\Bigl(\sum_{k=1}^n|\alpha_k|\Bigr)^2
  \leq (1+n)\|z\|^2.
\end{displaymath}
Hence, applying the Cauchy-Schwarz inequality as before, we have
\begin{align*}
  \|\Phi&(x,y)z
  -Df(y)z\|\\
  &=\Bigl\|\sum_{k=1}^n\beta_k
    \frac{f\bigl(y+\|x-y\|w_k\bigr)-f(y)
    -Df(y)\|x-y\|w_k}{\|x-y\|}\Bigr\|\\
  &\leq \sqrt{1+n}\|z\|\sqrt{\sum_{k=1}^n
    \left(\frac{\bigl\|f\bigl(y+\|x-y\|w_k\bigr)-f(y)
    -Df(y)\|x-y\|w_k\bigr\|}{\|x-y\|}\right)^2}
  \to 0
\end{align*}
as $x\to y$ by differentiability of $f$ at $y$. We conclude that
$\Phi(x,y)$ is separately continuous at every point $(x,x)$.

\begin{remark}
  (a) If $n=2$ there is a natural choice for $(v_1,v_2)$, but not for
  $n>2$. We choose $v_2$ to be the rotation of $v_1$ by $\pi/2$. More
  precisely, if $v_1=(z_1,z_2)$ we let $v_2=(-z_2,z_1)$.

  (b) The slope function $\Phi(x,y)$ constructed above is separately
  continuous at every point $(x,x)$. One could ask whether or not it is
  possible to choose it to be continuous at every $(x,y)$ with
  $x\neq y$. In our particular construction continuity is guaranteed if
  $(v_2,\dots,v_n)$ is continuous as a function of
  $v_1=\dfrac{y-x}{\|y-x\|}$. This is equivalent to finding $n-1$
  linearly independent solutions to the equation
  $\langle v_1,w\rangle=0$ depending continuously on $v_1$. Sufficient
  conditions for that are established in \cite{eckmann:43:sll}, and
  explicit orthonormal bases are given for dimensions $n=2$, $4$ and
  $8$. As shown in \cite{adams:62:vfs}, these are the only
  possibilities!  If $n\leq 8$ we can construct a slope function $\Phi$
  that is globally separately continuous if we artificially look at $f$
  as a function of $8$ variables by making it constant in $8-n$
  variables, and then restrict the constructed slope function to $n$
  variables just like a partial derivative; see
  Proposition~\ref{prop:partial-derivatives}. We do not claim that the
  construction of a globally separately continuous slope function is
  impossible for $n>8$, but only that some other method is required if
  it can be done.
\end{remark}

\paragraph{Lipschitz continuous functions.}
Let $E,F$ be Banach spaces and $U\subseteq E$ open. Recall that a
function $f\colon U\to F$ is called Lipschitz continuous if there exists
$L>0$ such that
\begin{equation}
  \label{eq:lipschitz}
  \|f(x)-f(y)\|_F\leq L\|x-y\|_E
\end{equation}
for all $x,y\in U$. We call $L$ a Lipschitz constant of $f$.

\begin{proposition}[Derivatives of Lipschitz functions]
  \label{prop:lipschitz}
  Let $E,F$ be Banach spaces and $U\subseteq E$ open. Assume that
  $f\colon U\to F$ is differentiable at $x\in U$. If $f$ is Lipschitz
  continuous with Lipschitz constant $L$, then
  $\|Df(x)\|_{\mathcal L(E,F)}\leq L$.
\end{proposition}
\begin{proof}
  Assume that $f$ is Lipschitz continuous with Lipschitz constant
  $L$. Let $\Phi$ be a slope function for $f$ at $x$. Then, for $z\in E$,
  we have
  \begin{displaymath}
    \|\Phi(x+tz)tz\|
    =\|f(x+tz)-f(x)\|
    \leq L\|tz\|
  \end{displaymath}
  whenever $t>0$ is small enough. Dividing by $t$ and then letting
  $t\to 0+$, we obtain
  \begin{displaymath}
    \|Df(x)z\|
    =\lim_{t\to 0+}\|\Phi(x+tz)z\|
    \leq L\|z\|
  \end{displaymath}
  for all $z\in E$. By definition of the operator norm
  $\|Df(x)\|_{\mathcal L(E,F)}\leq L$.
\end{proof}
Note that the converse is true when $U$ is convex. Indeed, by the mean
value inequality in Theorem~\ref{thm:mvt-ineq}, for every $x,y\in U$
there exists $c\in\lBrack x,y\rBrack$ such that
\begin{displaymath}
  \|f(y)-f(x)\|_F\leq\|Df(c)(y-x)\|_F
  \leq L\|y-x\|_E.
\end{displaymath}

\section{Application: Differentiable dependence of fixed points}
\label{sec:fixed-points}
The aim of this section is to use our approach to derivatives to give a
conceptually simple proof of the differentiable dependence of fixed
points in the Banach Fixed Point Theorem. The theorem is known, see for
instance \cite[Section~1.2.6]{henry:81:gts} or
\cite{hale:69:ode,brooks:09:cmp}.

Let $E$, $F$ be Banach spaces and let $U\subseteq E$ and
$\Lambda\subseteq F$ be non-empty open sets. Let
$f\colon \bar U\times\Lambda\to\bar U$ be a \emph{uniform contraction}
in $x\in U$. More precisely, assume that there exists $L\in(0,1)$ such
that
\begin{equation}
  \label{eq:uniform-contraction}
  \|f(x,\lambda)-f(y,\lambda)\|_E\leq L\|x-y\|_E
\end{equation}
for all $x,y\in \bar U$ and all $\lambda\in\Lambda$. By the Banach fixed
point theorem, for every $\lambda\in\Lambda$ there exists a unique fixed
point $x_\lambda\in\bar U$.

\begin{proposition}[Continuous dependence of fixed points]
  Assume that $f\colon\bar U\times\Lambda\to E$ satisfies
  \eqref{eq:uniform-contraction} with $L<1$. For every $\mu\in\Lambda$,
  let $x_\mu\in\bar U$ be the unique fixed point of $f(\cdot\,,\mu)$. If
  $\lambda\in\Lambda$ is such that $\mu\mapsto f(x_\lambda,\mu)$ is
  continuous at $\lambda$, then the map $\Lambda\to\bar U$,
  $\mu\mapsto x_\mu$ is continuous at $\lambda$.
\end{proposition}
\begin{proof}
  Using the assumption that $f$ is a uniform contraction, we have
  \begin{align*}
    \|x_\mu-x_\lambda\|
    &=\|f(x_\mu,\mu)-f(x_\lambda,\lambda)\|\\
    &\leq\|f(x_\mu,\mu)-f(x_\lambda,\mu)\|
      +\|f(x_\lambda,\mu)-f(x_\lambda,\lambda)\|\\
    &\leq L\|x_\mu-x_\lambda\|
      +\|f(x_\lambda,\mu)-f(x_\lambda,\lambda)\|.
  \end{align*}
  As $0<L<1$, by the continuity of $\mu\mapsto f(x_\lambda,\mu)$ at
  $\lambda$,
  \begin{displaymath}
    \|x_\mu-x_\lambda\|
    \leq\frac{1}{1-L}\|f(x_\lambda,\mu)-f(x_\lambda,\lambda)\|\to 0
  \end{displaymath}
  as $\mu\to\lambda$.
\end{proof}
We next show that the fixed points $x_\lambda$ depend differentiably on
$\lambda$. The reader is invited to compare our proof to a proof based
on Fr\'echet derivatives given, for instance, in
\cite[Section~1.2.6]{henry:81:gts}. By exploiting continuity properties
of the slope function, we can avoid all $\varepsilon$-$\delta$ arguments
and provide a conceptually cleaner proof.

\begin{theorem}[Differentiable dependence of fixed points]
  Assume that $f\in C^1(\bar U\times\Lambda,E)$ satisfies
  \eqref{eq:uniform-contraction} with $L<1$. For every $\mu\in\Lambda$,
  let $x_\mu\in\bar U$ be the unique fixed point of
  $f(\cdot\,,\mu)$. Then the map $\Lambda\to\bar U$, $\mu\mapsto x_\mu$
  is continuously differentiable.
\end{theorem}
\begin{proof}
  The idea is to use algebraic manipulations to find a slope function
  for the fixed points. If $\Phi$ is a slope function for $f$ and
  $x_\lambda$, $x_\mu$ are fixed points, then
  \begin{displaymath}
    \begin{split}
    x_\mu-x_\lambda
    &=f(x_\mu,\mu)-f(x_\lambda,\lambda)\\
    &=f(x_\mu,\mu)-f(x_\lambda,\mu)+f(x_\lambda,\mu)-f(x_\lambda,\lambda)\\
    &=\Phi\bigl((x_\lambda,\mu),(x_\mu,\mu)\bigr)(x_\mu-x_\lambda,0)
      +\Phi\bigl((x_\lambda,\lambda),(x_\lambda,\mu)\bigr)(0,\mu-\lambda).\\
    &=\Phi_1\bigl((x_\lambda,\mu),x_\mu\bigr)(x_\mu-x_\lambda)
      +\Phi_2\bigl((x_\lambda,\lambda),\mu\bigr)(\mu-\lambda),
    \end{split}
  \end{displaymath}
  where $\Phi_1$ and $\Phi_2$ are the partial slope functions for the
  functions $x\mapsto f(x,\lambda)$ and $\lambda\mapsto f(x,\lambda)$
  respectively, as introduced in
  Proposition~\ref{prop:partial-derivatives}.  Rearranging we see that
  \begin{displaymath}
    \bigl[I-\Phi_1\bigl((x_\lambda,\mu),x_\mu)\bigr)\bigr]
    (x_\mu-x_\lambda)
    =\Phi_2\bigl((x_\lambda,\lambda),\mu\bigr)(\mu-\lambda).
  \end{displaymath}
  Since $f$ is continuously differentiable on $\bar U\times\Lambda$,
  Theorem~\ref{thm:continuous-diff} allows us to choose $\Phi$ to be
  jointly continuous at
  $\bigl((\lambda,x_\lambda),(\lambda,x_\lambda)\bigr)$. Hence, as
  $L\in(0,1)$ and $\mu\mapsto x_\mu$ is continuous,
  Proposition~\ref{prop:lipschitz} implies the existence of $\delta>0$
  such that
  $\bigl\|\Phi_1\bigl((x_\lambda,\mu),x_\mu\bigr) \bigr\|_{\mathcal
    L(E)}<1$ whenever $\|\lambda-\mu\|<\delta$. Thus
  $\bigl[I-\Phi_1\bigl((x_\lambda,\mu),x_\mu\bigr)\bigr]^{-1}$ exists by
  a Neumann series expansion; see for instance
  \cite[Theorem~IV.1.4]{taylor:80:ifa}. Hence, if
  $\|\mu-\lambda\|<\delta$,
  \begin{displaymath}
    x_\mu=x_\lambda
    +\bigl[I-\Phi_1\bigl((x_\lambda,\mu),x_\mu\bigr)\bigr]^{-1}
    \Phi_2\bigl((x_\lambda,\mu),\lambda\bigr)(\mu-\lambda).
  \end{displaymath}
  Due to the joint continuity of $\Phi$ at
  $\bigl((\lambda,x_\lambda),(\lambda,x_\lambda)\bigr)$ and the
  continuity of inversion, we conclude that $\mu\mapsto x_\mu$ is
  differentiable at $\lambda$ with slope function given by
  \begin{displaymath}
    \Psi(\lambda,\mu)\gamma
    :=\bigl[I-\Phi_1\bigl((x_\lambda,\mu),x_\mu\bigr)\bigr]^{-1}
    \Phi_2\bigl((x_\lambda,\mu),\lambda\bigr)\gamma
  \end{displaymath}
  for all $\gamma\in F$ and derivative
  \begin{math}
    \Psi(\lambda,\lambda)
    =\bigl[I-D_xf(x_\lambda,\lambda)\bigr]^{-1}
    D_\lambda f(x_\lambda,\lambda)\in\mathcal L(F,E).
  \end{math}
\end{proof}

\pdfbookmark[1]{\refname}{biblio}%

\end{document}